\documentclass[12pt]{amsart}

\usepackage{mystyle}

\begin{document}

\title{Hölder regularity for non-variational porous media type equations}

\author[H. A. Chang-Lara]{H\'ector A. Chang-Lara}
\address{Department of Mathematics, CIMAT, Guanajuato, Mexico}
\email{hector.chang@cimat.mx}

\author[M. S. Santos]{Makson S. Santos}
\address{Department of Mathematics, CIMAT, Guanajuato, Mexico}
\email{makson.santos@cimat.mx}

\begin{abstract}
We present a Krylov-Safonov theory approach for the Hölder regularity of viscosity solutions to non-variational porous media type equations.  We explore the peculiarity of this type of problem: either the equation falls in a uniformly elliptic regime or the eikonal mechanism takes care of the regularity.  Our techniques are based on sliding paraboloids resulting in an ABP-type measure estimate.  By combining such estimates,  a diminishing of oscillation property is available, resulting in a regularity control in Hölder spaces.

\end{abstract}



\maketitle

\section{Introduction}

Consider the evolution problem modeled by a continuous and non-negative function $u\in\cC(B_1\times(-1,0])$ and driven by an equation of the form
\begin{equation}\label{eq:main1}
\p_t u = \sum_{i,j=1}^n ua_{ij}\p_{ij}u + b_{ij}\p_iu\p_ju  \;\;\mbox{ in }\;\; B_1\times(-1,0]
\end{equation}
where $a_{ij}=a_{ij}(x),b_{ij}=b_{ij}(x) \in \R^{n\times n}$ are symmetric matrices that satisfy the uniform ellipticity hypothesis
\[
\sum_{i,j=1}^n a_{ij}\xi_i\xi_j,\sum_{i,j=1}^n b_{ij}\xi_i\xi_j \in[\l,\L]\ss(0,\8) \qquad \forall \, \xi \in\p B_1
\]
(denoted by $(a_{ij}),(b_{ij}) \in [\l,\L]\ss(0,\8)$ from now on). The degeneracy of the diffusion as $u$ goes to zero is certainly the most interesting feature of these equations from the regularity point of view. This is responsible of the finite speed of propagation for the support of the solution, and gives one of the fundamental examples in free boundary problems.

In this article we pursue the regularity theory for these type of equations. To illustrate the variety of challenges encountered, consider the solution $(e\cdot x+t)_+$ for some $e\in\R^n$ such that $\sum_{i,j=1}^n b_{ij}e_ie_j=1$ with $(b_{ij})$ constant. This traveling front exhibits two interesting behaviors: The impossibility of a Harnack inequality, and the fact that we should not expect solutions to be better than Lipschitz. Here is our main theorem.

\begin{theorem}%\label{theo_main}
Given $[\l,\L]\ss(0,\8)$ there exist $\a \in (0,1)$ and $C>0$ such that the following holds: Let $(a_{ij}),(b_{ij}):B_1\times(-1,0] \to \R^{n\times n}$ symmetric and such that $(a_{ij}),(b_{ij}) \in [\l,\L]$, and $u \in {\mathcal C}(B_1\times(-1,0])$ be a solution of \eqref{eq:main1} taking values in $[0,1]$. Then 
\[
\|u\|_{{\mathcal C}^\a(B_{1/2}\times(-1/2,0])} \leq C.
\]
\end{theorem}

The classical example to keep in mind is the porous media equation (PME), historically developed from material sciences and fluid dynamics. Here $u$ is a density being transported by a potential flow and modeled by the continuity equation $\p_t u - \div(uDp) = 0$. The pressure $p$ gets related to the density $u$ by a constitutive relation which takes the form $p=\phi(u)$, in particular if $p=u$ we get the Boussinesq equation $\p_t u = u\D u + |Du|^2$.

Besides the classical motivation in fluid dynamics, the PME appears in other very interesting settings such as biological models \cite{Gurtin-1977,padron-2004} and differential geometry \cite{Wu-1993,MR1623198}. For a detailed exposition of the motivating problems we recommend the first chapters in the book by Vázquez \cite{vazquez-2007}.

% A nowadays very active model in the research community considers instead a convolution type constitutive relation of the form $p = \Phi \ast u$, which is natural from the point of view of mean field models with non-local interactions. For a detailed exposition of the classical models we recommend the first chapters in the book by Vázquez \cite{vazquez-2007}, meanwhile for the non-local problems we suggest the survey by Giacomin, Lebowitz, and Presutti \cite{Giacomin-1999}.

Due to its diverse applications and fascinating nonlinear structure, the PME and its generalizations have attracted the attention of many authors over the years. Most of the developments account for distributional weak solutions, including unique solvability, regularity, finite speed of propagation, and asymptotic behavior. The regularity of solutions in several dimensions was established in 1979 by Caffarelli and Friedman \cite{Caffarelli_Friedman-1979,Caffarelli_Friedman-1980} based on apriori estimates due to Aronson and Bénilan \cite{MR524760} around the same time. The theory was then quickly extended to the two phase Stefan problem by Caffarelli and Evans \cite{MR683353} and very general singular equations in divergence form by Ziemer \cite{MR654859}, DiBenedetto \cite{DiBenedetto-1982,DiBenedetto-1983}, and Sacks \cite{MR696738}; all of them around 1982 and based on the de Giorgi-Nash-Moser approach. Besides Chapter 7 in the book of Vázquez \cite{vazquez-2007}, we also recommend the books by DiBenedetto, Urbano, and Vespri \cite{DiBenedetto_Urbano_Vespri-2004}, and Urbano \cite{urbano-2008} for a detailed and pedagogical analysis on the regularity theory of degenerate equations in divergence form.

Meanwhile the degeneracy under the variational structure is well understood, the non-variational counterpart still offers some open questions. In this case the natural setting for the existence, uniqueness, and stability theorems are the viscosity solutions developed also in the eighties and nineties. For the classical PME, this treatment was started by Caffarelli and Vázquez in 1996 \cite{Caffareli_Vazquez-1999} and extended by Brändle and Vázquez in 2005 \cite{Brandle_Vazquez-2005}.

This work continues this line of research by analyzing the regularity of solutions to PME type equations in non-divergence form. In this sense our methods belong to the Krylov-Safonov regularity theory, originally presented in a probabilistic setting in \cite{Krylov-1979,Krylov-1980}. The main distinction between the variational and non-variational approach is that the energy estimates are replaced by the Aleksandrov-Bakelman-Pucci (ABP) principle. Standard references are the book by Gilbarg and Trudinger \cite{MR1814364} or the one by Caffarelli and Cabré \cite{Caffarelli_Cabre-1995}.

Let us emphasize that because of the degeneracy, we cannot expect a Harnack's inequality to hold in this scenario. However, as $u$ approaches zero, the evolution is driven by an eikonal equation which controls the speed of propagation of the level sets. We like to think about it in the following terms, either the equation falls in a uniformly elliptic regime or the eikonal mechanism takes care of the regularity. This is also, in broad sense, the point of view taken for the regularity theory of the Stefan problem in \cite{MR683353}.

To make this idea precise, we take into account the scale invariance of the equation and pursue a diminish of oscillation lemma. For a solution $u$ taking values in $[0,1]$ we consider two alternative scenarios in measure: Either $\{u\leq 1/2\}$ covers a positive fraction of the domain or $\{u> 1/2\}$ takes most of it instead. The first case leads to an improvement of the oscillation from above and it is the easiest to handle; in this instance $\max(u,1/2)$ is a sub-solution driven by a uniformly elliptic operator. For the second case we show that $u$ becomes positive in a smaller cylinder (and from then on falls in a uniformly elliptic regime) by proving an ABP-type measure estimate adapting Mooney's clever approach in \cite{Mooney-2015} to degenerate parabolic equations.

In the preliminary Section \ref{sec:pre} we review the notion of viscosity solution, give a precise statement of our theorem, and finally state the lemmas that build our result. Section \ref{sec:abp} takes care of the improvement from below, meanwhile in Section \ref{sec:ia} we do the improvement from above. The main theorem is finally proved in Section \ref{sec:pf}. In the concluding Section \ref{sec:fin} we recapitulate the whole strategy and comment on some further extensions.

\textbf{Acknowledgment:} Both authors were supported by CONACyT-MEXICO Grant A1-S-48577.

\section{Preliminaries}\label{sec:pre}

%\subsection{Scaling}
%
%The equation \eqref{eq:main1} enjoys the following scaling property which determines the geometry of our problem: For any $\b\in\R$ we get that if $u$ is a solution of \eqref{eq:main1}, then
%\[
%v(x,t) = r^{-\b}u(rx, r^{2-\b}t)
%\]
%is also a solution of an equation with similar control for the coefficients. This will also be denoted as the ellipticity family.

\subsection{Notation}

The \emph{open ball} in $\mathbb{R}^n$ of radius $r>0$ centered at $x\in \R^n$ is denoted by $B_r(x) := \{y\in\R^n:|y-x|<r\}$ and by default $B_r:=B_r(0)$. The \emph{open cube} of length $l>0$ and centered at $x=(x_1,\ldots,x_n)\in\R^n$ is
\[
Q_l(x) := (x_1-l/2,x_1+l/2)\times\ldots\times(x_n-l/2,x_n+l/2), \qquad Q_l := Q_l(0).
\]

A set $D\ss \R^n\times \R$ is open with respect to the \emph{parabolic topology} if for every $(x_0,t_0) \in D$ there exists some $r>0$ such that $B_r(x_0)\times(t_0-r,t_0]\ss D$. We call any set $N\ss \R^n\times \R$ that contains a cylinder of the form $B_r(x_0)\times(t_0-r,t_0]$ (for some $r>0$) a \emph{parabolic neighborhood} of $(x_0,t_0)$.

For a measurable set $S\ss\R^n\times\R$ (or perhaps $S \ss \R^n$) we denote the Lebesgue measure by $|S|$. Occasionally we may also use this notation for the Hausdorff measure of some lower dimensional set, the corresponding dimension for the measure should be clear from the context.

%We define the \emph{parabolic cylinder} by
%\[
%Q_r(x,t) := B_r(x)\times (t-r^{2-\beta},t]\ss\mathbb{R}^{n+1}, \qquad Q_r := Q_r(0,0).
%\]
%Notice that these cylinders are the balls coming from the \emph{distance}
%\[
%\dist((x,t),(y,s)):= \begin{cases}
%\max(|x-y|,(t-s)^{\frac{1}{2-\b}}) \text{ if } t\geq s\\
%\8\text{ otherwise}
%\end{cases}
%\]
%We will usually consider cylindrical domains of the form $Q=\W\times(t_i,t_f]$, in Section \ref{sec:meas} we will also consider cubes

Let $\a\in(0,1]$. A function $u:\W\times(t_i,t_f]\ss\R^n\times\R\rightarrow\mathbb{R}$ belongs to the \emph{H\"older space} $\mathcal{C}^{\a}(\W\times(t_i,t_f])$ if
\[
\|u\|_{\mathcal{C}^{\a}(\W\times(t_i,t_f])}:= \|u\|_{\mathcal C^0(\W\times(t_i,t_f])}+[u]_{\mathcal{C}^{\a}(\W\times(t_i,t_f])} <\8 \]
where
\[
\|u\|_{\mathcal C^0(\W\times(t_i,t_f])} := \sup_{\W\times(t_i,t_f]} |u|
\]
and
\[
[u]_{\mathcal{C}^{\a}(\W\times(t_i,t_f])}:= \sup_{\substack{(x,t),(y,s)\in \W\times(t_i,t_f] \\ (x,t)\neq(y,s)}}\frac{|u(x,t)-u(y,s)|}{(|x-y|+|t-s|)^{\a}}.
\]
We say that $u \in \mathcal C^\a_{loc}(\W\times(t_i,t_f])$ if $u \in \mathcal C^\a(B_r(x_0)\times(t_0-r,t_0])$ for any parabolic cylinder such that $B_{2r}(x_0)\times(t_0-2r,t_0] \ss \W\times(t_i,t_f]$.

Given $[\l,\L]\ss(0,\8)$ we define the Pucci's extremal operators acting on a symmetric matrix  $M = (M_{ij})\in\R^{n\times n}$ as
\begin{align*}
\cM_{\l,\L}^+(M) &:=\sup_{(a_{ij}) \in [\l,\L]}\sum_{i,j=1}^n a_{ij}M_{ij} =\sum_{e \in \eig (M)}(\L e_+ - \lambda e_-), \\
\cM_{\l,\L}^-(M) &:= \inf_{(a_{ij}) \in [\l,\L]}\sum_{i,j=1}^n a_{ij}M_{ij} =\sum_{e \in \eig (M)}(\l e_+ - \L e_-),
\end{align*}
where $\eig(M)$ is the set of eigenvalues of $M$ and $e_\pm = \max(\pm e,0)$ are the positive and negative parts of the real number $e$.

Notice that a function $u$ satisfies \eqref{eq:main1} for some coefficients $(a_{ij}),(b_{ij}) \in [\l,\L]$ if and only if
\begin{equation}\label{ineq}
u\cM_{\l,\L}^-(D^2u) + \l|Du|^2 \leq \p_t u \leq u\cM_{\l,\L}^+(D^2u) + \L|Du|^2.
\end{equation}

\begin{remark}\label{scaling}
These inequalities form a family of equations which is invariant under horizontal translations and the following scaling transformation: Let $\a\in \R$, if $u$ satisfies any of the two inequalities in \eqref{ineq} in a domain $\W\times(t_i,t_f]$ then $v(x,t) = r^{-\a}u(rx,r^{\b}t)$ also satisfies the corresponding inequality in $r^{-1}\W\times (r^{-\b}t_0,r^{-\b}t_f]$ provided that $\b=2-\a$.
\end{remark}

%They also allow us to consider equations of the form
%\[
%\p_t u = F(D^2u,u,x,t) + H(Du,x,t)
%\]

\subsection{Viscosity solutions}

In this section $F = F(M,p,z)$ will be a continuous function representing the fully-nonlinear operator driving the equation $\p_tu = F(D^2u,Du,u)$ over a domain $D\ss\R^n\times \R$, open in the parabolic sense. We should mainly keep in mind $F = |z|\cM_{\l,\L}^\pm M + b|p|^2$ (for $\l,\L,b\geq 0$), however other non-linearities will also be relevant (for example in the proof of Lemma \ref{lem:ia} in Section \ref{sec:ia}). Our goal is to define a weak notion of solution for this dynamic.

%Given a non-negative function $u\in \cC(\W\times(t_i,t_f])$ we define its \emph{positivity} set as $\W_+:= \{u>0\}\cap \W\times(t_i,t_f]$ and its \emph{free boundary} as $\G := \p\{u>0\}\cap \W\times(t_i,t_f]$. If there is the need to remind the dependence on the function we may indicate it in the following way $\W_+(u), \G(u)$, or $(\W_+\cup\G)(u)$. 

\begin{definition}
Let $D\ss\R^n\times \R$ a parabolic open set and $u \in \cC(D)$ be a non-negative function with $\p \spt u(\cdot,t) \cap D$ a smooth family of smooth surfaces, and $u$ is itself also smooth over $\spt u \cap D$. We say that $u$ is a classical sub-solution of
\[
\p_t u = F(D^2u,Du,u) \text{ in } D,
\]
if it satisfies in the classical point-wise sense
\[
\p_t u \leq F(D^2u,Du,u) \text{ in } \spt u \cap D.
\]
\end{definition}

We say that the sub-solution is \emph{strict} if the inequality is also strict at every point.

(Strict) classical super-solutions are defined in a similar way but with respect to the opposite inequalities.

The notion of viscosity solutions is given in terms of the \emph{comparison principle} that we should certainly expect for operators of the form $F(M,p,z) = |z|\cM_{\l,\L}^\pm M + b|p|^2$. To be precise, a (classical and non-negative) sub-solution $u$ that starts smaller than a (classical and non-negative) super-solution $v$ at some time, and never cross it over the lateral boundary of the domain, must remain smaller in the interior at future times. Otherwise we get a contradiction at the first contact point by using the comparison principle. Let us make these ideas rigorous with a couple definitions.

\begin{definition}
Let $(x_0,t_0)\in \R^{n}\times \R$ and $u$, $v$ be a pair of functions, continuously defined over some common parabolic neighborhood $N$ of $(x_0,t_0)$. We say that $v$ touches $u$ from above at $(x_0,t_0)$ over $N$, if $u\leq v$ with equality at $(x_0,t_0)$.
\end{definition}

Under the configurations of this definition, we may also say that $u$ touches $v$ from below.

\begin{definition}
Let $D\ss\R^n\times\R$ be a parabolic open set and $u \in \cC(D)$ a non-negative function. We say that $u$ is a viscosity sub-solution of 
\[
\p_t u = F(D^2u,Du,u)\text{ in } D,
\]
or that it satisfies
\[
\p_t u \leq F(D^2u,Du,u)\text{ in the viscosity sense in } D,
\]
if for every $(x_0,t_0)\in \spt u\cap D$, $u$ cannot be touched from above at $(x_0,t_0)$ by a strict classical super-solution of the same problem over a parabolic neighborhood of $(x_0,t_0)$ contained in $D$.
\end{definition}

The notion of \emph{viscosity super-solution} is given in a similar way by ruling out the contact from below with strict classical sub-solutions. Finally, \emph{viscosity solutions} are those which are simultaneously viscosity sub and super-solutions with respect to the same operator.

Most of the time our domains are just open cylinders for the parabolic topology, however in the proof of Lemma \ref{lem:abp} we also consider solutions in arbitrary open sets (see the domain for the equation \eqref{eq:ellip}).

This definition is consistent with the notion of classical solutions under the following ellipticity assumptions for $F:\mathbb R^{n\times n}\times \R^n\times \R\to \R$:

\textbf{(H1)} For any pair of symmetric matrices $M_1, M_2 \in \R^{n\times n}$, $p\in \R^n$ and $z\in \R$
\[
M_1\leq M_2 \qquad \Rightarrow\qquad F(M_1,p,z)\leq F(M_2,p,z).
\]

\textbf{(H2)} For any pair of symmetric matrices $M_1, M_2 \in \R^{n\times n}$ and $p_1,p_2\in \R^n$
\[
|p_1|\leq |p_2| \qquad \Rightarrow\qquad F(M_1,p_1,0)\leq F(M_2,p_2,0).
\]

In particular, notice that $F(M,p,z) = |z|\cM^\pm_{\l,\L}M+ b|p|^2$ with $\l,\L,b\geq 0$, satisfies (H1) and (H2) above.

We have the following consistency property between viscosity and classical solutions of elliptic problems.

\begin{property}
Let $D\ss\R^n\times \R$ be a parabolic open set, $F:\R^{n\times n}\times \R^n\times \R\to \R$ satisfy (H1) and (H2), and $u\in \mathcal C(D)$ be a non-negative function with $\p \spt u(\cdot,t) \cap D$ a smooth family of smooth surfaces, and $u$ is itself also smooth over $\spt u \cap D$. The following two are equivalent:
\begin{enumerate}
\item $\p_t u \leq F(D^2u,Du,u)$ in the classical sense in $D$.
\item $\p_t u \leq F(D^2u,Du,u)$ in the viscosity sense in $D$.
\end{enumerate}
\end{property}

\begin{proof}
$(1)\Rightarrow(2)$: Assume by contradiction that $u$ is a classical sub-solution and $\varphi$ is a strict classical super-solution that touches $u$ from above at $(x_0,t_0) \in \spt u\cap D$ over $B_r(x_0)\times(t_0-r,t_0] \ss D$. If $u(x_0,t_0)> 0$ we use the first and second derivative test together with (H1) to get that
\begin{align*}
\p_t \varphi(x_0,t_0) &\leq \p_t u(x_0,t_0),\\
&\leq F(D^2u(x_0,t_0),Du(x_0,t_0),u(x_0,t_0)),\\
&\leq  F(D^2\varphi(x_0,t_0),D\varphi(x_0,t_0),\varphi(x_0,t_0)).
\end{align*}
This contradicts $\varphi$ being a strict super-solution.

If instead $u(x_0,t_0)=0$, we use that for the positive parts of the functions $|D\varphi_+(x_0,t_0)|\geq |Du_+(x_0,t_0)|$  instead such that by (H2) we also get a contradiction.

$(2)\Rightarrow(1)$: Assume by contradiction that $u$ is a viscosity sub-solution and for some $(x_0,t_0) \in \spt u\cap D$ we have instead that
\[
\p_t u(x_0,t_0) > F(D^2u(x_0,t_0),Du(x_0,t_0),u(x_0,t_0)).
\]
By continuity we also have that the same inequality holds over $\spt u\cap N$ for some parabolic neighborhood $N:=B_r(x_0)\times (t_0-r,t_0] \ss D$. Then $\varphi=u$ defined over $N$ is a strict classical super-solution that touches $u$ and gives a contradiction.
\end{proof}

%It is also possible to extend the notion of viscosity solution for (possibly degenerate) equations with solutions that change sign, as for instance the problem
%\[
%\p_t u = |u|\D u + |Du_+|^2 - |Du_-|^2.
%\]
%In this case we consider $F=F(M,p_+,p_-,z)$ and instead of (H2) ask for the following:
%
%\textbf{(H2')} For any pair of symmetric matrices $M_1, M_2 \in \R^{n\times n}$ and $p_\pm,q_\pm\in \R^n$
%\[
%|p_+|\leq |q_+| \text{ and } |p_-|\geq |q_-| \qquad \Rightarrow\qquad F(M_1,p_+,p_-,0)\leq F(M_2,q_+,q_-,0).
%\]

The existence, uniqueness and stability of viscosity solutions to free boundary problems of the form $\p_tu = F(D^2u,Du,u)$ is a delicate issue that we will not pursue in this article. Whenever we invoke the comparison principle at least one of the solutions will be classical, so the conclusions just follow from the definitions. A careful analysis of the well-possedness by the viscosity approach has been studied for instance in \cite{Caffareli_Vazquez-1999,Brandle_Vazquez-2005} for equations of the form
\[
\p_t u = a(u)\D u+|Du|^2
\]
with even weaker assumption on the diffusion $a$ than we have (Brändle and Vázquez allow $a$ to be sub-linear as $u\to0^+$). Here we would like to highlight as well a closely related line of research for several free boundary problems such as Hele-Shaw and Stefan \cite{MR2763347}, Richards equation \cite{MR3116010}, or the PME with drift \cite{MR2600689}.

%\begin{proposition}[Diminishing of oscillation]\label{prop_osc} Let $u \in \cC(B_1\times(-1,0])$ such that $\osc_{Q_1}u \leq 1$. Suppose that there exist constants $\r,\b \in (0,1)$ such that the following property holds for all $Q_r(x,t) \subset Q_1$
%\begin{equation}\label{eq:oscproperty}
%\mbox{if } \osc_{Q_r(x,t)}u \leq r^\b, \mbox{ then }\;\osc_{Q_{\r r}(x,t)} u\leq (\r r)^\b.
%\end{equation}   
%Then $u \in \cC^\b(Q_{1/2})$ and there exits a universal constant $C >0$ such that
%\[
%\|u\|_{\cC^\b(Q_{1/2})} \leq C.
%\]
%\end{proposition}

\subsection{Main result and overview of the strategy}

Here is a precise and rigorous version of the theorem announced in the introduction.

\begin{theorem}\label{thm}
Given $[\l,\L]\ss(0,\8)$ there exists $\a \in (0,1)$ such that the following hold: Let $D \ss \R^n\times\R$ be a parabolic open set and $u \in {\mathcal C}(D)$ a non-negative function that satisfies the following inequalities in the viscosity sense
\[
\begin{cases}
\p_t u \geq u\cM_{\l,\L}^-(D^2u) + \l|Du|^2 \text{ in } D\\
\p_t u \leq u\cM_{\l,\L}^+(D^2u) + \L|Du|^2 \text{ in } D
\end{cases}
\]
then $u \in \cC^\a_{loc}(D)$.
\end{theorem}

We will present a detailed proof of this theorem in the Section \ref{sec:pf}. Here is a sketch of the main steps: By the translation and scale invariance of the problem (Remark \ref{scaling}) it suffices to establish a decay of oscillation for the solutions over parabolic cylinders, this is provided by the next two lemmas. We consider two possible alternatives for $u$ in measure in order to discriminate in which direction does the oscillation improve. First we state the most delicate case, whether $u$ is a non-negative super-solution that is larger than $1/2$ in $B_1\times(-1,0]$ over a sufficiently large fraction of the cylinder. In this case the oscillation improves from below and then the equation becomes uniformly parabolic. The following lemma is proven in Section \ref{sec:abp}.

\begin{lemma}[Improvement from below] \label{lem:ib}
Given $[\l,\L]\ss(0,\8)$ there exist $\eta,\theta \in (0,1)$ such that if $u \in {\mathcal C}(B_1\times(-1,0])$ is a non-negative function that satisfies
\[
\p_t u \geq u\cM_{\l,\L}^-(D^2u) + \l|Du|^2 \text{ in the viscosity sense in } B_1\times(-1,0],
\]
then
\[
\frac{|\{u>1/2\}\cap B_1\times(-1,0]|}{|B_1\times(-1,0]|} \geq (1-\eta) \qquad \Rightarrow\qquad \inf_{B_{1/2}\times(-1/2,0]} u\geq \theta.
\]
\end{lemma}

Once fixed the fraction $\eta \in (0,1)$ by the previous lemma we can consider the alternative scenario, namely $u$ being less than $1/2$ in $B_1\times(-1,0]$ at least a small fraction of the cylinder. Section \ref{sec:ia} contains the proof of the following result.

\begin{lemma}[Improvement from above]\label{lem:ia}
Given $[\l,\L]\ss(0,\8)$ and $\eta>0$, there exists $\theta \in (0,1)$ such that if $u \in {\mathcal C}(B_1\times(-1,0])$ takes values in $[0,1]$ and satisfies
\[
\p_t u \leq u\cM_{\l,\L}^+(D^2u) + \L|Du|^2 \text{ in the viscosity sense in } B_1\times(-1,0],
\]
then
\[
\frac{|\{u\leq 1/2\}\cap B_1\times(-1,0]|}{|B_1\times(-1,0]| } \geq \eta\qquad \Rightarrow\qquad \sup_{B_{1/2}\times(-1/2,0]} u\leq (1-\theta). 
\]
\end{lemma}

\section{Improvement from below}\label{sec:abp}

In this section we prove a contrapositive version of Lemma \ref{lem:ib} in a convenient geometric configuration. At the end of this section we will see how to obtain Lemma \ref{lem:ib} from this result by a scaling and covering argument.

\begin{lemma}[Measure estimate]\label{lem:ib2}
Given $[\l,\L]\ss(0,\8)$ there exist $M,\eta >0$ such that if $u \in \cC(B_{10}\times(-1,1])$ is a non-negative functions that satisfies
\[
\p_t u \geq u\cM_{\l,\L}^-(D^2u) + \l|Du|^2 \text{ in the viscosity sense in } B_{10}\times(-1,1],
\]
then
\[
u(0,1) \leq 1 \qquad\Rightarrow\qquad |\{u\leq M\}\cap B_{10}\times(-1,1]| \geq \eta.
\]
\end{lemma}

The technique to establish this type of measure estimates in the uniformly elliptic or parabolic regime is known as the ABP Lemma. It consists on showing that there is a set of positive measure where $u$ can be touched by a family of paraboloids.

\subsection{Paraboloids and contact sets}

A concave \emph{paraboloid} centered at the \emph{vertex} $(x_0,t_0) \in\R^n\times\R$, of opening $\a > 0$ in space, and slope $\b>0$ in time, is the quadratic function
\[
P^{\a,\b}_{x_0,t_0}(x,t) = P^{\a,\b}_{(x_0,t_0)}(x,t) := -\tfrac{\a}{2}|x-x_0|^2+\b(t-t_0)
\]
%In order to simplify the notation we may use
%\[
%P^\a_{x_0,t_0} := P^{\a,\a}_{x_0,t_0},\qquad P^{\a,\b} := P^{\a,\b}_{0,0}, \qquad P^\a := P^{\a,\a}_{0,0}.
%\]

\begin{remark}\label{scaling2} The paraboloids are related by the Lipschitz (in space and time) scaling transformation 
\[
P^{r\a,\b}_{x_0,t_0}(x,t) = r^{-1} P^{\a,\b}_{rx_0,rt_0}(rx,rt)
\]
Keep in mind that by Remark \ref{scaling}, this Lipschitz transformation also leaves invariant the equations under consideration. 
\end{remark}

%Now we are about to define the \emph{contact set} of a non-negative continuous function $u:\W\times (t_i,t_f)\to\R$. In order to do so we ask that for every vertex in $B \Subset \W\times (t_i,t_f)$ the corresponding paraboloid have contact with the solution in the interior of the domain of the equation. The first hypothesis below rules our contact at the boundary and the second guarantees that there is some contact point. Of course these properties will have to be checked or included in the hypotheses of our lemmas:
%\begin{enumerate}
%\item\label{h1} $\bigcup_{(x,t) \in B} \spt (P_{x,t}^{\a,\b}(\cdot,t_f))_+ \Subset \W$,
%\item\label{h2} $\forall \, (x,t)\in B$, $\exists \, (y,s) \in \W\times(t,t_f]$ such that $P_{x,t}^{\a,\b}$ touches $u$ from below at $(y,s)$.
%\end{enumerate}

We define the \emph{contact set} of $u \in \cC(\W\times(t,s])$ with respect to a set of vertices $V\ss \R^n\times \R$ as
\begin{align*}
A^{\a,\b}_V(\W,t,s) := \{(x^*,t^*) \in \W\times (t,s] \ : \ &\text{$\exists \, (x_0,t_0)\in V$ such that $P_{x_0,t_0}^{\a,\b}$ touches $u$}\\ &\text{from below at $(x^*,t^*)$ over $\W\times(t,t^*]$}\}
\end{align*}
We may omit the dependence on $\W$, $t$, or $s$, whenever these are clear from the context. We also denote $A^{\a,\b}_{x_0,t_0} := A^{\a,\b}_{\{(x_0,t_0)\}}$.
%\[
%A^{\a,\b}_{x_0,t_0} := A^{\a,\b}_{\{(x_0,t_0)\}},\qquad A^\a_B := A^{\a,\a}_{B}, \qquad A^{\a,\b} := A^{\a,\b}_{0,0}, \qquad A^\a := A^{\a,\a}_{0,0}.
%\]

\subsection{Aleksandrov-Bakelman-Pucci-type lemma}

%In the following lemma we fix $u \in \mathcal C(B_R\times(0,1])$, and the parameters
%\[
%\a:=100\max(1,\l), \qquad m:=\min(1,1/(n\L)).
%\]
%The following fact for the set $A^{\a,\a}_{0,0} := A^{\a,\a}_{0,0}(u,B_R\times(0,1])$ is used in the statement of the lemma and its own proof:
%\[
%A^{\a,\a}_{0,0} \ss \spt (P^{\a,\a}_{0,0})_+ \cap \R^n \times [0,1] \ss B_{\sqrt 2}\times[0,1]
%\]

%Notice that under the hypothesis $u(0,1)\leq 1$ from Lemma \ref{lem:ib2} we get that $A^{\a,\a}_{0,0} \neq \emptyset$. Hence, the the main point of the hypothesis in the next lemma is whether at least one of the points in $A^{\a,\a}_{0,0}$ falls in the uniformly elliptic regime.

%Notice for the moment that from $\t\geq 2n$ we get the localization
%\[
%A^{\a,\a}_{0,0}\ss B_{1/4}\times(0,1]
%\]
%which we have to be non-empty under the assumption $u(0,1)\leq 1$ from Lemma \ref{lem:ib2}.

\begin{lemma}[ABP]\label{lem:abp}
Given $[\l,\L] \ss (0,\8)$, $\a>1$, and $m\in(0,1)$, there exist $M,\eta>0$ such that the following holds: Let $\t \in (0,3]$ and $u \in \cC(B_{10} \times(0,\t])$ be a non-negative function that satisfies
\[
\p_t u \geq u\cM_{\l,\L}^-(D^2u) \text{ in the viscosity sense in } B_{10} \times(0,\t],
\]
then
\[
A^{\a,\a}_{0,0}(B_{10},0,\t) \cap \{u > m\} \neq \emptyset \qquad\Rightarrow\qquad |\{u\leq M\}\cap B_{10}\times(0,\t]| \geq \eta.
\]
\end{lemma}

\begin{proof}
Assuming that
\[
(x^*_0,t_0^*) \in A^{\a,\a}_{0,0}(0,\t)\cap \{u > m\} \text{ exists,}
\]
we will construct a set $V\ss\R^n\times[0,\8)$ of vertices such that for each $(x_0,t_0) \in V$
\[
\emptyset \neq A^{2\a,2\a}_{x_0,t_0}(t_0^*-r^2,t_0^*) \ss \{u \in [m/2,M]\},
\]
where
\[
r := m/(100\a) \in (0,1), \qquad M:= 6\a.
\]

Given that in the time interval $(0,\t]$, any paraboloid $P^{2\a,2\a}_{x_0,t_0}$ (with $t_0\geq 0$) is less than or equal to $2\a\tau \leq 6\a$, we get for free that $A^{2\a,2\a}_{V}(t_0^*-r^2,t_0^*) \ss \{u\leq M\}$. The proof will be finished once we get a lower bound on the measure of the contact set $A^{2\a,2\a}_{V}(t_0^*-r^2,t_0^*)$, which will be a consequence of the uniform ellipticity of the equation over $\{u \in [m/2,M]\}$.

Before going into the construction of $V$, let us point out a few facts about the contact point $(x_0^*,t_0^*)$. Because $t^*_0\leq \t\leq 3$ we have that
\[
x_0^* \in \overline{\{P_{0,0}^{\a,\a}(\cdot,t_0^*)>0\}} \ss \overline{\{P_{0,0}^{\a,\a}(\cdot,3)>0\}} = \overline{B_{\sqrt 6}}.
\]
Also
\[
\a t_0^* \geq P_{0,0}^{\a,\a}(x_0^*,t_0^*) > m \geq \a r^2,
\]
which gives us that $(t_0^*-r^2,t_0^*]\ss(0,\t]$. Putting these two together and recalling that $r\in(0,1)$ we observe that
\[
N := B_{4r}(x^*_0)\times (t_0^*-r^2,t_0^*] \ss B_{7}\times(0,\t].
\]

\begin{figure}
    \centering
    \includegraphics[scale=.7]{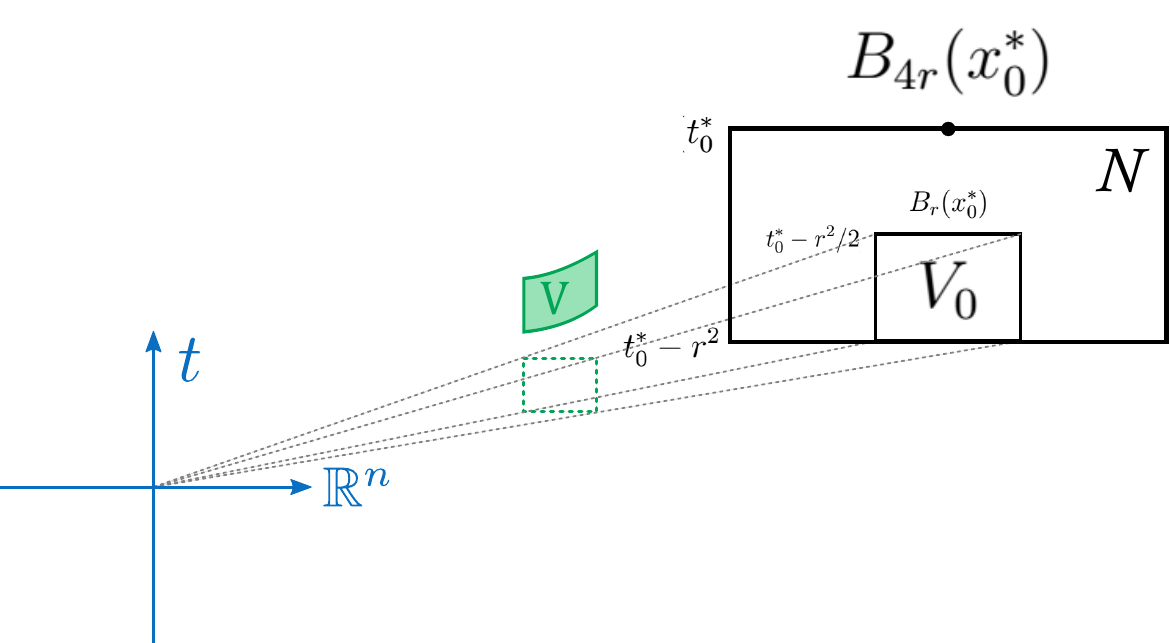}
    \caption{Configuration for the vertices of the paraboloids around the contact point $(x_0^*,t_0^*)$}
    \label{fig2}
\end{figure}

Let
\[
V := \Phi(V_0), \qquad \Phi(x,t) := (x/2, t/2+|x|^2/8), \qquad V_0 := B_{r}(x^*_0)\times (t_0^*-r^2,t_0^*-r^2/2],
\]
see Figure \ref{fig2}. The idea is that for any $(x_0,t_0) \in V_0$ we have that
\[
P^{\a,\a}_{0,0} + P^{\a,\a}_{x_0,t_0} = P^{2\a,2\a}_{\Phi(x_0,t_0)}.
\]
Notice that for $t\in (0,\t]$ we still have the compact inclusion of the supports of the (positive parts) of the paraboloids
\begin{align}\label{eq:sop}
\overline{\{ P^{2\a,2\a}_{\Phi(x_0,t_0)}(\cdot,t)>0\}} \ss \overline{\{(P^{\a,\a}_{0,0} + P^{\a,\a}_{x_0,t_0})(\cdot,3)>0\}} \ss \overline{B_{2\sqrt{6}+r}} \ss B_7.
\end{align}

Let us show first that for $(x_0,t_0) \in V_0$
\[
\emptyset \underset{\textbf{(1)}}{\neq} A^{2\a,2\a}_{\Phi(x_0,t_0)}(t_0^*-r^2,t_0^*) \underset{\textbf{(2)}}{\ss} N \underset{\textbf{(3)}}{\ss} \{u\geq m/2\}.
\]

\textbf{(1)} On one hand, $P^{2\a,2\a}_{\Phi(x_0,t_0)} = P^{\a,\a}_{0,0} + P^{\a,\a}_{x_0,t_0} < u$ over the time interval $(0,t_0^*-r^2]$ because $P^{\a,\a}_{0,0}\leq u$ and $P^{\a,\a}_{x_0,t_0}<0$ over the same interval. On the other hand, 
\[
P^{\a,\a}_{x_0,t_0}(x^*_0,t_0^*) = -\tfrac{\a}{2}|x^*_0-x_0|^2 + \a(t_0^*-t_0) \geq -\tfrac{\a}{2}r^2 + \a \tfrac{r^2}{2} = 0.
\]
This means that $P^{2\a,2\a}_{\Phi(x_0,t_0)}(x^*_0,t_0^*) \geq u(x^*_0,t_0^*)$, so we have shown that the paraboloid $P^{2 \a,2\a}_{\Phi(x_0,t_0)}$ reaches $u$ at some point in the time interval $(t_0^*-r^2,t_0^*]$ (and necessarily in $B_{10}$ by \eqref{eq:sop}), i.e.
\[
\emptyset \neq A^{2\a,2\a}_{\Phi(x_0,t_0)}(t_0^*-r^2,t_0^*).
\]

\textbf{(2)} For $(x,t) \in (\R^n\sm B_{4r}(x^*_0))\times (t_0^*-r^2,t_0^*]$
\[
P^{\a,\a}_{x_0,t_0}(x,t) \leq -\tfrac{\a}{2}9r^2 + \a r^2 < 0.
\]
Hence, $P^{2\a,2\a}_{\Phi(x_0,t_0)} < P^{\a,\a}_{0,0} \leq u$ over $(\R^n\sm B_{4r}(x^*_0))\times(t_0^*-r^2,t_0^*]$. In other words
\[
A^{2\a,2\a}_{\Phi(x_0,t_0)}(t_0^*-r^2,t_0^*) \ss N.
\]

\textbf{(3)} Finally for any $(x,t) \in N$,
\begin{align*}
P_{0,0}^{\a,\a}(x,t) &= P_{0,0}^{\a,\a}(x_0^*,t_0^*) -\a\1 x^*_0\cdot (x-x^*_0) + \tfrac{1}{2}|x-x^*_0|^2 -(t-t_0^*)\2,\\
&> m - \a(4\sqrt 6 r + 8r^2 + r^2),\\
&\geq m - 21\a r,\\
&\geq m/2.
\end{align*}
We used that $x_0 \in B_{\sqrt 6}$ and $r = m/(100\a) \in (0,1)$ in the previous estimates. Given that $u\geq P_{0,0}^{\a,\a}$ over $N$ we conclude that
\[
N \ss \{u\geq m/2\}.
\]

Now we present the second part of the proof which consists on estimating (for some $\eta>0$ to be fixed)
\[
|A^{2\a,2\a}_{V}(t_0^*-r^2,t_0^*)| \geq \eta.
\]
As already announced, this depends on the fact that $u$ satisfies
\begin{align}\label{eq:ellip}
\p_t u \geq \cM_{m\l/4,2M\L}^-(D^2u) \text{ in the viscosity sense in } \{u\in (m/4,2M)\},
\end{align}
a fact that can be checked automatically from the definition.

From now on we will assume, without loss of generality, that $u$ is semi-concave in space and time, otherwise regularize using the inf-convolution (see for instance \cite[Lemma 4.2]{MR3158522} which refers back to \cite[Section 8]{MR1118699}). Under this regularity we can define (up to a set of measure zero) the map from the contact points $(x^*,t^*) \in A^{2\a,2\a}_{V}(t_0^*-r^2,t_0^*)$ to the unique vertex $(x_0,t_0) \in V$ such that $P_{x_0,t_0}^{2\a,2\a}$ touches $u$ from below at $(x^*,t^*)$
\begin{align*}
(x^*,t^*) \mapsto (x_0,t_0) &:= (x^* + \tfrac{1}{2\a}Du(x^*,t^*), t^* - \tfrac{1}{2\a}u(x^*,t^*) - \tfrac{1}{8\a^2}|Du(x^*,t^*)|^2).
\end{align*}
Indeed, these formulas comes from the fact that at the contact point we have that
\[
Du(x^*,t^*) = -2\a(x^*-x_0), \qquad\text{ and } \qquad u(x^*,t^*) = -\tfrac{1}{4\a}|Du(x^*,t^*)|^2 + 2\a(t^*-t_0).
\]
The properties proven on the previous steps about the contact set, together with the semi-concavity of $u$, guarantee that the map is well defined and surjective.

The determinant of the Jacobian can be computed as
\begin{align*}
&\det \begin{pmatrix} 
I+\tfrac{1}{2\a}D^2 u & \tfrac{1}{2\a}D\p_t u \\
-\tfrac{1}{2\a}Du^T - \tfrac{1}{4\a^2}Du^TD^2 u & 1-\tfrac{1}{2\a}\p_t u - \tfrac{1}{4\a^2}Du\cdot D\p_t u
\end{pmatrix},\\
= \, &\det \begin{pmatrix} 
I+\tfrac{1}{2\a}D^2 u & \tfrac{1}{2\a}D\p_t u \\
0 & 1-\tfrac{1}{2\a}\p_t u
\end{pmatrix},\\
= \, &(1-\tfrac{1}{2\a}\p_t u)\det(I+\tfrac{1}{2\a}D^2u).
\end{align*}
At the contact set we have that $D^2 u \geq -2\a I$, and $\p_t u \leq 2\a$. Therefore the determinant above is non-negative.

By invoking the area formula for Lipschitz maps we get that
\[
2^{n-1}r^{n+2}|B_1| = |V| = \iint_{A^{2\a,2\a}_{V}(t_0^*-r^2,t_0^*)}(1-\tfrac{1}{2\a}\p_t u)\det(I+\tfrac{1}{2\a}D^2u).
\]
We used that $V = \Phi(V_0)$ with $\det(D\Phi) = 2^{-(n+1)}$ and $|V_0|=(4r)^nr^2|B_1|$.

Recall that over $\{u\in (m/4,2M)\}$, $u$ is driven by a uniformly elliptic operator and then we can use the equation to bound the second order derivatives from above over the contact set
\[
2\a \geq \p_t u \geq \cM_{m\l/4,M\L}(D^2 u) = \sum_{e\in \eig(D^2u)} \1\tfrac{m\l}{4}e_+ - 2M\L e_-\2 \geq -4M\L\a.
\]
By applying this estimate to the area formula we get the desired inequality
\[
2^{n-1}r^{n+2}|B_1| \leq (1+2M\L)(1+4/(m\l))^n|A^{2\a,2\a}_{V}(t_0^*-r^2,t_0^*)|.
\]
In other words, we can just take
\[
\eta := \frac{2^{n-1}r^{n+2}|B_1|}{(1+2M\L)(1+4/(m\l))^n} = \frac{2^{n-1}(m/(100\a))^{n+2}|B_1|}{(1+12\a\L)(1+4/(m\l))^n}.
\]
\end{proof}

The previous proof can also be adapted to the following version needed in the iterative procedure for the next section. Keep in mind that the equations are invariant by Lipschitz rescalings, see Remark \ref{scaling} for $\a=1$ and Remark \ref{scaling2} for the paraboloids.

\begin{corollary}\label{cor:abp}
Given $[\l,\L] \ss (0,\8)$, $\a>1$, and $m\in(0,1)$, there exist $M,\eta>0$ such that the following holds: Let $\t \in (0,3]$, $k \in \Z_{\geq0}$, and $u \in \cC(B_{2^{-k}10} \times(0,2^{-k}\t])$ be a non-negative function that satisfies
\[
\p_t u \geq u\cM^-_{\l,\L}(D^2u) \text{ in the viscosity sense in } B_{2^{-k}10} \times(0,2^{-k}\t],
\]
then
\[
A^{2^k\a,\a}_{0,0}(B_{2^{-k}10},0,2^{-k}\t) \cap \{u > 2^{-k}m\} \neq \emptyset \qquad\Rightarrow\qquad \frac{|\{u\leq 2^{-k}M\} \cap B_{2^{-k}10} \times(0,2^{-k}\t]|}{|B_{2^{-k}10} \times(0,2^{-k}\t]|} \geq \eta.
\]
\end{corollary}

\begin{proof}
Given $k\in \Z_{\geq0}$ and $u$ satisfying the hypotheses with respect to this $k$, let $u_k(x,t) := 2^ku(2^{-k}x,2^{-k}t)$, which satisfies the hypotheses of Lemma \ref{lem:abp} thanks to Remark \ref{scaling} and Remark \ref{scaling2}. Hence for $M$ and $\bar\eta$ as in Lemma \ref{lem:abp},
\begin{align*}
\frac{|\{u\leq 2^{-k}M\} \cap B_{2^{-k}10} \times(0,2^{-k}\t]|}{|B_{2^{-k}10} \times(0,2^{-k}\t]|} &= \frac{|\{u_k \leq M\} \cap B_{10} \times(0,\t]|}{|B_{10} \times(0,\t]|} \geq \frac{\bar\eta}{3|B_{10}|} =: \eta.
\end{align*}
\end{proof}

\subsection{Measure estimate at a fixed time}

The previous ABP-type lemma only applies if we know that the contact point falls on the uniformly elliptic regime. In this section we lift such requirement by implementing a selection algorithm iterated over dyadic cubes. The result of this procedure is (in the worst case) a measure estimate in space and at time $t=-1$.

%In what follows, we establish some measure estimates. In order to avoid technicalities, we assume that $\osc u \leq 1$. The general case can be done by using a scaling argument. The next lemma states that if $u$ is above $1/2$ ``most of the time'', then we can separate $u$ from zero. 
%
%\begin{lemma}[Measure estimate]\label{lem:impbel}
%Let $u \in {\mathcal C}(B_1\times(-1,0])$ be a viscosity super-solution to
%\[
%\partial_tu \geq u{\mathcal M}^-u + \lambda|Du|^2 \;\;\mbox{ in }\;\;B_1\times(-1,0].
%\]
%Then
%\[
%\inf_{B_{\r}\times(-\r,0]} u \leq m \qquad\Rightarrow\qquad |\{u > 1/2\}\cap B_{1}\times\left(-1,0\right]| \geq 1-\m,
%\]
%for some constants $m,\r,\m>0$ depending on $\l,\L$, and $n$. 
%\end{lemma}

%Instead of proving directly Lemma \ref{lem:impbel}, it is more convenient to  rephrase it the following Lemma.  
%
%\begin{lemma}[Measure Estimate]\label{lem:me}
%Let $u \in \cC(B_{3\sqrt{n}}\times(-2n-1/2,1])$ be a non-negative viscosity super-solution
%\[
%\p_t u \geq u\cM^-(D^2u) + \l|Du|^2 \text{ in } B_{3\sqrt{n}}\times(-2n-1/2,1].
%\]
%If $P^\a$ touches over $\{u<m\} \cap B_{1/4}\times[0,1/2]$, then
%\[
%|\{u\leq M\}\cap B_{3\sqrt{n}}\times[-2n-1/2,-2n]| \geq \m
%\]
%For some universal $\a,m,M,\m>0$.
%\end{lemma}

%Once we prove Lemma \ref{lem:me}, the proof of Lemma \ref{lem:impbel} is completed by using standards scaling/covering arguments.
%
%The idea of the proof is to take a family of cylinders in possibly every dyadic scale, where a local measure estimate holds thanks to the ABP-type Lemma \ref{lem:abp}. In the following section we explain how to select these cylinders in detail.

\subsubsection{Decomposition}

Consider the dyadic decomposition of the cube
\[
Q_{1/\sqrt{n}} = (-1/(2\sqrt n),1/(2\sqrt n))^n \ss B_{10}
\]
into $2^n$ congruent and disjoint sub-cubes $Q^1_j := Q_{1/(2\sqrt n)}(x_j^1)$ where
\[
x_j^1 := \frac{1}{4\sqrt{n}}(j_{1,1},\ldots,j_{1,n}) \text{ with $j := (j_{1,1},\ldots,j_{1,n}) \in \{\pm 1\}^{1\times n}$.}
\]

The $k^{th}$ generation consists of $2^{kn}$ congruent and disjoint sub-cubes $\{Q^k_j := Q_{1/(2^k\sqrt n)}(x_j^k)\}$ generated from a similar decomposition for each cube in the previous generation. In this general case we can label the cubes, or their centers, using matrices of the form $j = (j_{\a,\b}) \in \{\pm 1\}^{k\times n}$ such that we have that each coordinate of the center $x_j^k$ is given by
\[
x_j^k \cdot e_\b := \frac{1}{2\sqrt{n}}\sum_{\a=1}^{k} 2^{-\a}j_{\a,\b}.
\]
For $k=0$ we consider by default $j=\emptyset$, $x_\emptyset^0=0$, and $Q^0_\emptyset=Q_{1/\sqrt n}$.

We say that $j \in \{\pm1\}^{k\times n}$ is a direct descendant of  $j' \in \{\pm1\}^{(k-1)\times n}$ if the cube $Q_j^k$ is one of the $2^n$ sub-cubes obtained from the decomposition of the cube $Q_{j'}^{k-1}$. In other words, $j'$ consists of the first $(k-1)$ rows of $j$ (i.e. $j' = (j_{\a,\b})_{\a\leq (k-1)}$). Reciprocally, we may just say that $j'$ is the progenitor of $j$. An ancestor of $j$ is in this way any sub-matrix $j' = (j_{\a,\b})_{\a\leq l} \in \{\pm 1\}^{l\times n}$ for $l\in[0,k)$ integer.

Let $\t\in (1/2,1]$. For $k\geq 0$, we consider the time intervals of the form $(t_k,s_k]$ with
\[
t_k := 2^{-k}-1, \qquad s_k := \min(2^{-(k-2)}-1,\t).
\]

See Figure \ref{fig1} for an illustration of the dyadic decomposition and the time intervals that will be considered in our constructions.

\begin{figure}[t]
\centering
\includegraphics[width=10cm]{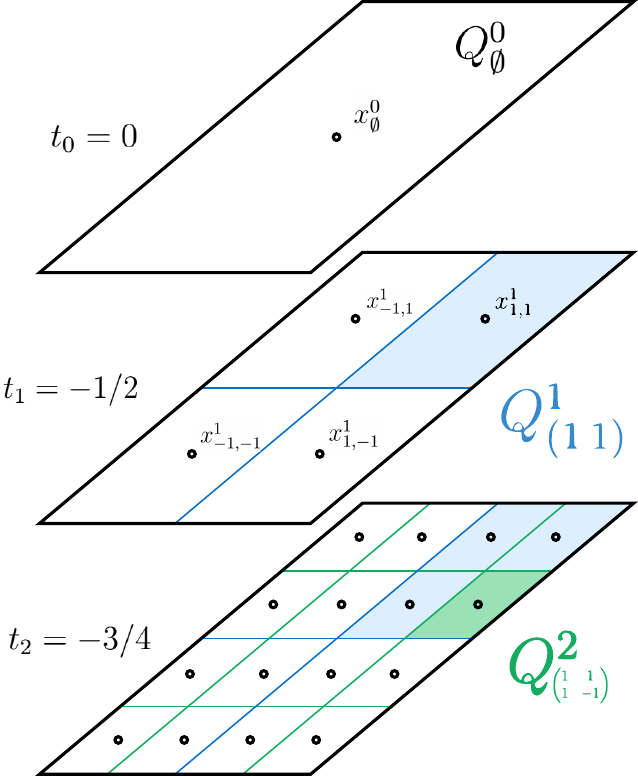}
\caption{Dyadic decomposition.}
\label{fig1}
\end{figure}

%The idea is that for $k\geq 1$ we have that:
%\begin{itemize}
%\item The length of the interval $(t_k,s_k]$ decays geometrically: $s_k - t_k = 3\times 2^{-k}$.
%\item The intervals $(t_k,s_{k-1}]$ and $(t_i^{k},s_{k}]$ have as intersection the interval $(t_k,s_{k}]$ of length $2\times 2^{-k}$.
%\item $t_k,s_k\to -1$ as $k\to \infty$.
%\item The interval $(s_k-(2-2\t) = (4\t-2)2^{-k}$.
%\end{itemize}
%By choosing $\t=3/2$ we get that the intervals accumulate towards the time $(2-2\t)=-1$.

The selection algorithm is described in the following statement. In this lemma we finally fix the parameters $\a$ and $m$, therefore also the constants $M$ and $\eta$ appearing in the previous lemma and corollary.

\begin{lemma}[Dyadic decomposition]\label{lem:dec}
Given $[\l,\L]\ss(0,\8)$ there exist $\a>1$ and $m\in(0,1)$ such that the following holds: Let $\t \in(1/2,1]$ and $u \in \cC(B_{10}\times[-1,\t])$ be a non-negative function that satisfies
\[
\p_t u \geq u\cM_{\l,\L}^-(D^2u) + \l|Du|^2 \text{ in the viscosity sense in } B_{10}\times(-1,\t],
\]
such that $A_{0,0}^{\a,\a}(B_{10},0,\t) \neq 0$. Starting with $k=0$, we define the sets of indices
\begin{align*}
\mathcal G_k := \{j\in\{\pm1\}^{k\times n} \ | \ &(j_{a,\b})_{\a\leq l} \notin \mathcal G_l \text{ for any integer } l \in [0,k),\\
&A_{x_j^k,t_k}^{2^k\a,\a}(B_{2^{-k}10}(x_j^k),t_k,s_k) \cap\{u>2^{-k}m\}\neq \emptyset\}.
\end{align*}
Then
\[
Q_{1/\sqrt n} \sm \bigcup_{k\geq 0} \bigcup_{j\in \mathcal G_k} Q_j^k \ss \{u(\cdot,-1)=0\}.
\]
\end{lemma}

The first condition in the construction of $\mathcal G_k$ says that no ancestor of $j$ has been previously chosen in $\mathcal G_l$ for $l\in[0,k)$ integer. The second condition will be used to apply Corollary \ref{cor:abp} over the domain $B_{2^{-k}10}(x_j^k)\times(t_k,s_k]$.

\begin{proof}
Let us notice from the very beginning that for any $k\geq 0$ and $j \in \{\pm1\}^{k\times n}$, the positive part of the paraboloid $P_{x_j^k,t_k}^{2^k\a,\a}$ is supported inside $B_{2^{-k}\sqrt 6}(x_j^k) \ss B_4\ss B_{10}$ during the time interval $(t_k,s_k]$. Indeed, for $t\in (t_k,s_k]$
\[
\overline{\{ P_{x_j^k,t_k}^{2^k\a,\a}(\cdot,t)>0\}} \ss \overline{\{ P_{x_j^k,t_k}^{2^k\a,\a}(\cdot,s_k)>0\}} = \overline{B_{\sqrt{2^{-(k-1)}(s_k-t_k)}}(x_j^k)} \ss \overline{B_{2^{-k}\sqrt 6}(x_j^k)}.
\]

Let $j \in \{\pm 1\}^{k\times n}$ such that for any integer $l\in[0,k)$ it holds that $(j_{\a,\b})_{\a\leq l}\notin \mathcal G_l$. Let us show then that
\begin{align}\label{eq:ind}
A_{x_j^k,t_k}^{2^k\a,\a}(B_{2^{-k}10}(x_j^k),t_k,s_k) \neq \emptyset.
\end{align}

Once this gets established we notice that for any $x\in Q_{1/\sqrt n}\sm \bigcup_{k\geq 1} \bigcup_{j\in \mathcal G_k} Q_j^k$ there exists a sequence $j^{(k)} \in \{\pm 1\}^{k\times n}\sm \mathcal G_k$ such that $j^{(k-1)}$ is the progenitor of $j^{(k)}$ and $x_{j^{(k)}}^k \to x$ as $k\to \8$. By \eqref{eq:ind} and the construction of $\mathcal G_k$ we necessarily have that
\[
\emptyset \neq A_{x_{j^{(k)}}^k,t_k}^{2^k\a,\a}(B_{2^{-k}10}(x_{j^{(k)}}^k),t_k,s_k) \ss \{u\leq 2^{-k}m\},
\]
which means that
\[
\{u\leq 2^{-k}m\}\cap B_{2^{-k}10}(x_{j^{(k)}}^k)\times(t_k,s_k] \neq \emptyset.
\]
Hence by the continuity of $u$ and the fact that the cylinders accumulate towards $(x,-1)$, we deduce that $u(x,-1)=0$, the desired conclusion of the lemma.

The identity \eqref{eq:ind} is certainly true for $k=0$ by the hypothesis $A_{0,0}^{\a,\a}(B_{10},0,\t) \neq 0$. Assume then inductively that for the progenitor $j'\in \{\pm 1\}^{(k-1)\times n}$ of $j\in \{\pm 1\}^{k\times n}$ it holds that
\[
A_{x_{j'}^{k-1},t_k}^{2^{k-1}\a,\a}(B_{2^{-(k-1)}10}(x_{j'}^{k-1}),t_{k-1},s_{k-1}) \neq \emptyset.
\]
Because $j'\notin \mathcal G_{k-1}$ we necessarily have that
\[
(x^*,t^*) \in A_{x_{j'}^{k-1},t_{k-1}}^{2^{k-1}\a,\a}(B_{2^{-(k-1)}10}(x_{j'}^{k-1}),t_{k-1},s_{k-1}) \cap \{u\leq 2^{-(k-1)}m\} \text{ exists}.
\]
We will see now that by fixing $\a$ large and $m$ small, $(x^*,t^*)$ necessarily localizes close to $(x_{j'}^{k-1},t_{k-1})$.

By using $\varphi := P_{x_{j'}^{k-1},t_{k-1}}^{2^{k-1}\a,\a}$ as a test function for the equation at the contact point $(x^*,t^*) \in \{u\leq 2^{-(k-1)}m\}$ we get that
\begin{align*}
\a &\geq \varphi(x^*,t^*)\cM^-_{\l,\L}(-2^{k-1}\a I) + \l 2^{2(k-1)}\a^2|x^*-x_{j'}^{k-1}|^2,\\
&\geq - nm\L\a + \l 2^{2(k-1)}\a^2|x^*-x_{j'}^{k-1}|^2,
\end{align*}
then
\[
\sqrt{\frac{1+nm\L}{4\l\a}} \geq 2^k|x^*-x_{j'}^{k-1}|.
\]
By finally fixing
\[
\a := 100\max(1,1/\l), \qquad  m := \min(1,1/(n\L)),
\]
we guarantee that
\[
x^*\in B_{2^{-k}/10}(x_{j'}^{k-1}).
\]
Because $\a\geq 2m$ we get that for $(x,t) \in B_{2^{-k}/10}(x_{j'}^{k-1})\times (s_k,\8)$
\begin{align*}
P_{x_{j'}^{k-1},t_{k-1}}^{2^{k-1}\a,\a}(x,t) &\geq -2^{k-2}\a(2^{-k}/10)^2 + \a \min(2^{-(k-1)},\t+1-2^{-(k-1)}),\\
&\geq \a 2^{-k}(2-1/400), \qquad\qquad (\t\geq 1/2)\\
&> m2^{-(k-1)}.
\end{align*}
This means that $t^* \in (t_{k-1},s_k]$.

As a final step let us check now that
\begin{align}\label{eq:ind2}
P_{x_j^k,t_k}^{2^k\a,\a}\geq 2^{-(k-1)}m \text{ in }B_{2^{-k}/10}(x_{j'}^{k-1})\times(t_{k-1},s_k].
\end{align}
Together with the localization of $(x^*,t^*)$ this would imply that $P_{x_j^k,t_k}^{2^k\a,\a}$ necessarily reaches $u$ in the time interval $(t_k,s_k]$. Recalling that for any $t\in (t_k,s_k]$ it holds that $\overline{\{ P_{x_j^k,t_k}^{2^k\a,\a}(\cdot,t)>0\}} \ss B_{2^{-k}10}(x_j^k)$, we would get from this step that \eqref{eq:ind} is true and conclude the proof.

Given $(x,t) \in B_{2^{-k}/10}(x_{j'}^{k-1})\times(t_{k-1},s_k]$
\begin{align*}
P_{x_j^k,t_k}^{2^k\a,\a}(x,t) &\geq -2^{k-1}\a(|x_{j'}^{k-1} - x_j^k| +2^{-k}/10)^2 + \a 2^{-k},\\
&= -2^{k-1}\a(2^{-(k-1)}/4 +2^{-(k-1)}/20)^2 + \a 2^{-(k-1)}/2,\\
&\geq 2^{-(k-1)}\a/10.
\end{align*}
Using that $\a>10m$ we settle the desired lower bound.
\end{proof}

\subsubsection{Fixed time measure estimate}

Let $u$ be as in Lemma \ref{lem:ib2}. Whenever we have that $j \in \mathcal G_k$, Corollary \ref{cor:abp} can be applied to $u$ over the domain $B_{2^{-k}10}(x_j^k)\times(t_k,s_k]$, as a result we get the following lower bound for the density of the set $\{u\leq M\}\supseteq \{u\leq 2^{-k}M\}$ in the cylinder $B_{2^{-k}10}(x_j^k)\times(t_k,s_k]$
\[
\frac{|\{u\leq M\}\cap B_{2^{-k}10}(x_j^k)\times(t_k,s_k]|}{|B_{2^{-k}10}(x_j^k)\times(t_k,s_k]|} \geq \eta.
\]
The disadvantage of these configurations is that as $k\to\8$ these cylinders converge to $t=-1$, effectively recovering an estimate in space and not in space-time as expected in Lemma \ref{lem:ib2}. However we will see in the following section that these estimates can be integrated in time to recover the desired bound. To give a rigorous proof of this estimate at a fixed time, we consider a projection of the sets in space. In order to do this we invoke the following geometric lemma, see Figure \ref{fig3}.

\begin{figure}[t]
\centering
\includegraphics[width=14cm]{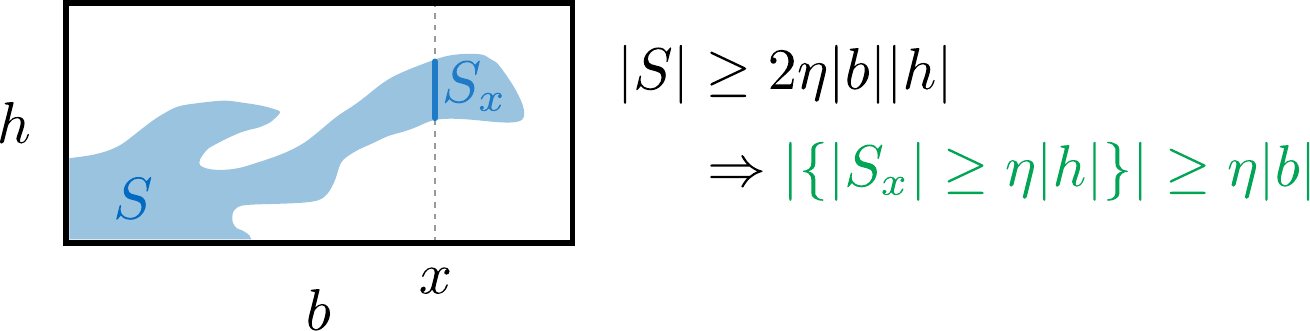}
\caption{An elementary geometric fact.}
\label{fig3}
\end{figure}

\begin{lemma}\label{lem:geo}
Let $\eta>0$, $S \ss b\times h \ss \R^n\times \R$ with positive measure, and for $x\in b$ define the fiber of $S$ over $x$ as $S_x := \{t\in h\ | \ (x,t)\in S\}$. Then
\[
\frac{|S|}{|b\times h|} \geq 2\eta \qquad\Rightarrow\qquad \frac{|\{x \in b \ | \ |S_x|/|h|\geq \eta\}|}{|b|} \geq \eta.
\] 
\end{lemma}

\begin{proof}
By sub-additivity
\[
2\eta|b||h| \leq |S| \leq |\{x \in b \ | \ |S_x|\geq \eta|h|\}||h| + |\{(x,t) \in b\times h \ | \ |S_x|< \eta|h|\}|.
\]
By Fubini
\[
|\{(x,t) \in b\times h \ | \ |S_x|< \eta|h|\}| \leq \eta|b||h|.
\]
Hence
\[
|\{x \in b \ | \ |S_x|\geq \eta|h|\}||h| \geq 2\eta|b||h| - \eta|b||h| = \eta|b||h|.
\]
\end{proof}

As we plan to apply this to the set $S = \{u\leq M\}$ it is convenient to consider the density of $\{u\leq M\}$ for each $x \in B_{10}$ over a sub-interval $(t,s]\ss(-1,1]$
\[
\theta_M(x,t,s) := \frac{|\{u(x,\cdot)\leq M\}\cap (t,s]|}{s-t}.
\]
As well the maximal version for $t \in [-1,1]$
\[
\Theta_M(x,t) := \sup_{s\in(t,1]}\theta_M(x,t,s).
\]

\begin{corollary}
Let $\eta,M,R>0$, $(t,s]\ss\R$ and $u \in \mathcal C(B_R\times(t,s])$. Then
\[
\frac{|\{u\leq M\}\cap B_{R}\times(t,s]|}{|B_R\times(t,s]|} \geq \eta \qquad\Rightarrow\qquad 
\frac{|\{\theta_M(\cdot,t,s)\geq \eta/2\}\cap B_R|}{|B_R|}\geq \frac{\eta}{2}.
\]
\end{corollary}

\begin{proof}
Indeed, if we take $S:=\{u\leq M\}$ and $b\times h = B_R\times(t,s]$, then $|S_x|/|h| = \theta_M(x,t,s)$.
\end{proof}

From the conclusion of Corollary \ref{cor:abp},
\[
\frac{|\{u\leq M\}\cap B_{2^{-k}10}(x_j^k)\times(t_k,s_k]|}{|B_{2^{-k}10}(x_j^k)\times(t_k,s_k]|} \geq \eta,
\]
and using that the length of $(-1,s_k]$ is at most eight times the length of $(t_k,s_k]$, so that for $x\in B_{10}$
\[
\Theta_M(x,-1) \geq \theta_M(x,-1,t_k) \geq \frac{1}{8} \theta_M(x,s_k,t_k),
\]
we deduce that
\[
\frac{|\{\Theta_M(\cdot,-1)\geq \eta/16\}\cap B_{2^{-k}10}(x_j^k)|}{|B_{2^{-k}10}(x_j^k)|}\geq \frac{\eta}{2}.
\]

The following result is a corollary of the dyadic decomposition (Lemma \ref{lem:dec}), the ABP corollary (Corollary \ref{cor:abp}) and the previous geometric observations.

\begin{lemma}[Fixed time measure estimate]\label{cor:fxtm}
Given $[\l,\L] \ss (0,\8)$, let $\a=100\max(1,1/\l)$ be as in Lemma \ref{lem:dec} and $M=6\a$ as in Corollary \ref{cor:abp}. There exists $\eta\in (0,1)$ such that if $\t\in(1/2,1]$, $u \in \mathcal C(B_{10}\times[-1,\t])$ is a non-negative function that satisfies
\[
\p_t u \geq u\cM^-_{\l,\L}(D^2u) + \l|Du|^2 \text{ in the viscosity sense in } B_{10}\times(-1,\t],
\]
then $A^{\a,\a}_{0,0}(B_{10},0,\t)\neq \emptyset$ implies that at least one of the following alternatives hold
\[
\frac{|\{u(\cdot,-1)=0\} \cap Q_{1/\sqrt n}|}{|Q_{1/\sqrt n}|} \geq 1/2 \qquad\text{or}\qquad |\{\Theta_M(\cdot,-1)\geq \eta\}|\geq \eta.
\]
\end{lemma}

\begin{proof}
Let $\bar\eta$ be as in Corollary \ref{cor:abp} and $\mathcal G_k$ be as in Lemma \ref{lem:dec}.

If the first alternative is false, that is to say $|\{u(\cdot,-1)=0\} \cap Q_{1/\sqrt n}| < 1/2|Q_{1/\sqrt n}|$, we get, thanks to the Lemma \ref{lem:dec}, that
\[
\left| \bigcup_{k\geq 0} \bigcup_{j\in \mathcal G_k}  B_{2^{-k}10}(x_j^k)\right| \geq \left| \bigcup_{k\geq 0} \bigcup_{j\in \mathcal G_k}  Q_{2^{-k}}(x_j^k)\right| \geq \frac{1}{2}|Q_{1/\sqrt n}|.
\]

By Corollary \ref{cor:abp} and the geometric observations in this section we get that for each $j \in \mathcal G_k$
\[
\frac{|\{\Theta_M(\cdot,-1) > \bar\eta/16\}\cap B_{2^{-k}10}(x_j^k)|}{|B_{2^{-k}10}(x_j^k)|} \geq \frac{\bar\eta}{2}.
\]

By Vitali Covering Lemma we can then extract a subset $\mathcal G \ss \bigcup_{k\geq 0}\mathcal G_j$ such that the balls $\{B_{2^{-k}10}(x_j^k)\}_{j\in \mathcal G}$ form a disjoint set and
\[
\sum_{j\in \mathcal G}\left|B_{2^{-k}10}(x_j^k) \right| \geq 5^{-n}\left| \bigcup_{k\geq 0} \bigcup_{j\in \mathcal G_k}  B_{2^{-k}10}(x_j^k)\right|.
\]

By finally taking $\eta := (5^{-n}|Q_{1/\sqrt n}|/100)\bar\eta$ we get that
\begin{align*}
|\{\Theta_M(\cdot,-1)\geq \eta\}| &\geq \sum_{j\in \mathcal G}\left|\{\Theta_M(\cdot,-1)\geq \bar\eta/16\}\cap B_{2^{-k}10}(x_j^k) \right|,\\
&\geq \frac{\bar\eta}{2}\sum_{j\in \mathcal G}\left|B_{2^{-k}10}(x_j^k) \right|,\\
&\geq \frac{2\eta}{|Q_{1/\sqrt n}|} \left| \bigcup_{k\geq 0} \bigcup_{j\in \mathcal G_k}  B_{2^{-k}10}(x_j^k)\right|,\\
&\geq \eta.
\end{align*}
The desired conclusion of the lemma.
\end{proof}

\subsection{Integration in time} As a final step to get the measure estimate Lemma \ref{lem:ib2}, we need to apply Lemma \ref{cor:fxtm} in a whole interval of times and obtain a set of positive measure in \emph{space and time} where the super-solution is bounded.

\begin{lemma}\label{cor:ib3}
Given $[\l,\L] \ss (0,\8)$, let $\a=100\max(1,1/\l)$ be as in Lemma \ref{lem:dec} and $M=6\a$ as in Corollary \ref{cor:abp}. There exists $\eta\in(0,1)$ such that the following holds: Let $u \in \mathcal C(B_{10}\times(-1,1])$ be a non-negative function that satisfies
\[
\p_t u \geq u\cM^-_{\l,\L}(D^2u) + \l|Du|^2 \text{ in the viscosity sense in } B_{10}\times(-1,1].
\]
If $A^{\a,\a}_{0,\t}(B_{10},\t,1) \neq \emptyset$ for each $\t\in(0,1/2]$, then
\[
|\{u\leq M\}\cap B_{10}\times(-1,1]| \geq \eta.
\]
\end{lemma}

\begin{proof}
Let $\bar\eta$ be as in Lemma \ref{cor:fxtm}. By applying Lemma \ref{cor:fxtm} to $u(\cdot,\cdot+\t)$ we get that one of the following alternatives hold for each $\t\in (0,1/2]$
\[
\frac{|\{u(\cdot,-1+\t)=0\}\cap Q_{1/\sqrt n}|}{|Q_{1/\sqrt n}|} \geq \frac{1}{2} \qquad \text{or} \qquad |\{\Theta_M(\cdot,-1+\t) \geq \bar\eta\} \cap B_{10}| \geq \bar\eta.
\]

If the set of $\t\in (0,1/2]$ where the first alternative holds has length at least $1/4$ we conclude by Fubini: The function $u$ vanishes in a set of $(n+1)$-dimensional measure at least $\eta_1:= |Q_{1/\sqrt n}|/8$ in $Q_{1/\sqrt n}\times(-1,-1/2]\ss B_{10}\times(-1,1]$. So let us assume instead that the set of $\t\in (0,1/2]$ where the second alternative holds has measure at least $1/4$. By Fubini,
\[
|\{\Theta_M \geq \bar\eta\}\cap B_{10}\times(-1,-1/2]|\geq \frac{\bar\eta}{4}.
\]

By applying Lemma \ref{lem:geo} with $S := \{\Theta_M \geq \bar\eta\}$ in $b\times h := B_{10}\times(-1,-1/2]$ we get that for
\[
G:=\{x\in B_{10} \ | \ |S_x| \geq \bar\eta/(4|B_{10}|)\} \qquad\text{ and }\qquad S_x = \{t\in(-1,-1/2] \ | \ \Theta_M(x,t) \geq \bar\eta\},
\]
it holds that
\begin{align}\label{eq:ppp}
|G| \geq \frac{\bar\eta}{4|B_{10}|}.
\end{align}

Lastly, let us show that for each $x_0 \in G$
\[
|\{t\in (-1,1] \ | \ u(x_0,t)\leq M\}|\geq \frac{\bar\eta^2}{32|B_{10}|}.
\]
Together with the estimate \eqref{eq:ppp} it will settle the proof of the lemma with $\eta := \min(\eta_1,\eta_2)$ and $\eta_2:= \frac{\bar\eta^3}{128|B_{10}|^2}$.

Let $x_0 \in G$. By the definition of $\Theta_M$ there exists $\ell(t) \in (0,1-t]$ for each $t\in S_{x_0}$ such that
\[
\theta_M(x_0,t,s(t)) = \frac{|\{u(x_0,\cdot)\leq M\}\cap (t,t+\ell(t)]|}{\ell(t)} \geq \frac{\bar\eta}{2}.
\]

By applying Vitali Covering Lemma to the covering $\{(t-\ell(t),t+\ell(t))\}_{t\in S_{x_0}}$ of the fiber $S_{x_0}$ we get that there exists a countable set $\mathcal T \ss S_{x_0}$ such that $\{(t-\ell(t),t+\ell(t))\}_{t\in \mathcal T}$ is disjoint and
\[
\sum_{t\in \mathcal T} \ell(t) \geq \frac{1}{4}|S_{x_0}|.
\] 
Hence
\begin{align*}
\frac{\bar\eta}{4|B_{10}|} &\leq |S_{x_0}|,\\
&\leq 4\sum_{t\in \mathcal T} \ell(t),\\
&\leq \frac{8}{\bar\eta} \sum_{t\in\mathcal T}|\{u(x_0,\cdot) \leq M\}\cap (t,t+\ell(t)]|,\\
&\leq \frac{8}{\bar\eta} |\{u(x_0,\cdot)\leq M\}\cap (-1,1]|.
\end{align*}
the announced estimate with which we conclude the proof.
\end{proof}

\subsection{Proof of Lemma \ref{lem:ib2} and Lemma \ref{lem:ib}}

\begin{proof}[Proof of Lemma \ref{lem:ib2}]
Recall that $\a := 100\max(1,1/\l)$ was fixed in Lemma \ref{lem:dec} and from there we choose $M=6\a$ and $\eta$ from Lemma \ref{cor:ib3}. All we need to check is that $u(0,1)\leq 1$ implies $A_{0,\t}^{\a,\a}(B_{10},\t,1)\neq \emptyset$ for every $\t \in (0,1/2]$ in order to apply Lemma \ref{cor:ib3}.

Indeed, recall first that for every $t \in (\t,1]$ we observe that $\overline{\{ P_{0,\t}^{\a,\a}(\cdot,t)>0\}} \ss \overline{B_{\sqrt 2}}\ss B_{10}$. On the other hand, the paraboloid starts from zero at time $\t$ and at time $t=1$ we get that
\[
P_{0,\t}^{\a,\a}(0,1) = \a(1-\t) \geq \a/2 \geq 1.
\]
Then, definitely $P_{0,\t}^{\a,\a}$ reaches $u$ at some intermediate time inside $B_{10}$.
\end{proof}

\begin{proof}[Proof of Lemma \ref{lem:ib}]
Let $u\in \mathcal C(B_1\times(-1,0])$ satisfies the hypotheses from Lemma \ref{lem:ib} with respect to the parameters $[\l,\L]\ss(0,\8)$. Let $M=6\a=600\max(1,1/\l)$ and $\bar\eta>0$ the constants from Lemma \ref{lem:ib2} with respect to the same ellipticity parameters. Assume by contradiction that for $\theta:= 1/(2M)$ we have that $u(x_0,t_0) < \theta$ for some $(x_0,t_0) \in B_{1/2}\times(-1/2,0]$.

Consider the Lipschitz rescaling
\[
v(x,t) := 2M u(x/(2M)+x_0,(t-1)/2M+t_0)
\]
such that $v$ restricted to $B_{10}\times(-1,1]$ satisfies the super-solution equation from Lemma \ref{lem:ib2} (Remark \ref{scaling}).

Given that $u(x_0,t_0)<\theta$ is equivalent to say that $v(0,1) \leq 1$ we get from Lemma \ref{lem:ib2} that
\begin{align*}
\frac{|\{u \leq 1/2\}\cap B_1\times(-1,0]|}{|B_1\times(-1,0]|} &\geq \frac{5^n}{M^{n+1}}\frac{|\{u \leq 1/2\}\cap B_{5/M}(x_0) \times(t_0-1/M,t_0]|}{|B_{5/M}(x_0) \times(t_0-1/M,t_0]|},\\
&= \frac{5^n}{M^{n+1}}\frac{|\{v\leq M\}\cap B_{10}\times(-1,1]|}{|B_{10}\times(-1,1]|}\\
&\geq \frac{5^n\bar\eta}{M^{n+1}|B_{10}\times(-1,1]|}.
\end{align*}
Hence, if we finally choose $\eta$ to be the right-hand side above, we contradict that the set where $u$ is instead larger than $1/2$ is less than $1-\eta$, which settles the proof.
\end{proof}

\section{Improvement from above}\label{sec:ia}

As a last step towards the proof of Theorem \ref{thm} we need to establish the improvement of the upper bound provided a weak control in measure for a sub-solution.

\begin{proof}[Proof of Lemma \ref{lem:ia}]
The truncation $w := \max(1/2, u)$ is a viscosity sub-solution to a semi-linear parabolic equation
\[
\partial_tw \leq \cM_{\l/2,\L}^+(D^2 w) + \Lambda|Dw|^2 \;\;\mbox{ in }\;\;B_1\times(-1,0].
\]
This fact can be directly checked from the definition of viscosity sub-solutions.

Then we consider $v := e^{Aw}$ such that, again from the definition of sub-solution (by applying the same transformation to the test functions) we get that the following inequality holds in the viscosity sense
\[
\partial_tv\leq \cM_{\l/2,\L}^+\left(D^2v - \dfrac{Dv\otimes Dv}{v}\right) + \dfrac{\Lambda|Dv|^2}{Av}\;\;\mbox{ in }\;\;B_1\times(-1,0].
\]
By taking $A$ sufficiently large in terms of the ellipticity constants, we get to cancel the gradient terms in order to recover the caloric equation 
\[
\partial_tv \leq \cM_{\l/2,\L}^+(D^2v)\;\;\mbox{ in }\;\;B_1\times(-1,0].
\]
The result now follows by applying the weak Harnack's inequality to $(e^A-v)$.
\end{proof}

\section{Proof of the main theorem}\label{sec:pf}

By combining Lemma \ref{lem:ib} and Lemma \ref{lem:ia}, we are able to prove a diminish of oscillation property, and finish the proof of Theorem \ref{thm}.

\begin{proof}[Proof of Theorem \ref{thm}]
Given $[\l,\L]\ss(0,\8)$, let $\eta\in(0,1)$ be the fraction from the improvement from below (Lemma \ref{lem:ib}), $\theta \in (0,1)$ the smallest constant between those appearing in Lemma \ref{lem:ib} and Lemma \ref{lem:ia} (for the previous $\eta$), and $\b:= |\ln(1-\theta)/\ln 2|$.

Given $B_{r_0}(x_0)\times (t_0-r_0,t_0]\ss D$ with $r_0\in(0,1)$, let
\[
M := \max(1,\|u\|_{\mathcal C^0(B_{r_0/2}(x_0)\times (t_0-r_0/2,t_0])}).
\]
Our goal is to show that
\[
\sup_{\r\in (0,\min(r_0/2,M^{-1}))}\r^{-\b}\osc_{B_\r(x_0)\times (t_0-\r^{2-\b},t_0]} u < \8,
\]
from where we get that $u\in \mathcal C^{\a}_{loc}(D)$ for $\a=\b/(2-\b)$.

Consider the following rescaling for $r_1:=\min(r_0/2,M^{-1})\in(0,1)$
\[
v(x,t) := r_1 u(r_1 x+x_0,r_1^{3}t+t_0).
\]
This function is defined in $B_1\times(-1,0]$ and takes values between zero and one. It also satisfies the same inequalities in the viscosity sense as $u$ (Remark \ref{scaling} with exponents $\a=-1$ and $\b=3$). Our goal for this function is then to show that
\[
\sup_{\r\in (0,1)}\r^{-\b}\osc_{B_\r\times (-\r^{2-\b},0]} v < \8.
\]
We will actually show by induction that the the left-hand side gets bounded by $(1-\theta)^{-1}$. In order to do this we only need to consider the cases $\r=2^{-k}$ (for $k\in \Z_{\geq0}$) and show instead that
\begin{align}\label{eq:induction}
\osc_{B_{2^{-k}}\times (-2^{-(2-\b)k},0]} v \leq (1-\theta)^k =2^{-\b k}.
\end{align}
For $\r\in (2^{-(k+1)},2^{-k}]$ it then follows that
\[
\osc_{B_\r\times (-\r^{2-\b},0]} v \leq \osc_{B_{2^{-k}}\times (-2^{-(2-\b)k},0]} v \leq (1-\theta)^{-1}(1-\theta)^{k+1} = (1-\theta)^{-1}2^{-\b(k+1)} \leq (1-\theta)^{-1}\r^\b.
\]

Clearly we have that \eqref{eq:induction} holds for $k=0$. Assume then the hypothesis for some arbitrary $k\geq 0$ and consider the rescaling
\[
w(x,t) = 2^{\b k}v(2^{-k}x,2^{-(2-\b)k}t).
\]
It is non-negative and satisfies the same inequalities in the viscosity sense as $v$ over $B_1\times(-1,0]$ (Remark \ref{scaling}). Moreover, thanks to the inductive hypothesis
\[
\sup_{B_1\times(-1,0]} w \leq 2^{\b k}\osc_{B_{2^{-k}}\times (-2^{-(2-\b)k},0]} v \leq 2^{\b k}(1-\theta)^k = 1.
\] 

For the fraction $\eta\in(0,1)$, chosen from Lemma \ref{lem:ib}, we have the following alternatives:
\[
\frac{|\{w > 1/2\}\cap B_{1}\times\left[-1,0 \right)|}{|B_{1}\times\left[-1,0 \right)|} \geq (1-\eta) \qquad\text{ or }\qquad  \frac{|\{w \leq 1/2\}\cap B_{1}\times\left[-1,0 \right)|}{|B_{1}\times\left[-1,0 \right)|} \geq \eta.
\]
Hence either Lemma \ref{lem:ib} or \ref{lem:ia} indicate that
\[
\osc_{B_{1/2}\times(-2^{-(2-\b)},0]}w \leq \osc_{B_{1/2}\times(-1/2,0]}w \leq 1-\theta.
\]
For $v$ this is the inductive step required to conclude the proof.
\end{proof}

\section{Conclusion and further directions}\label{sec:fin}

In this article we have been able to establish the Hölder regularity for viscosity solutions to porous media type equations following the Krylov-Safonov theory. The degeneracy feature in our equation roughly says that the solution either falls in a uniformly elliptic regime (this is the role of the set $\{u>m\}$ in the ABP Lemma \ref{lem:abp}) or follows an eikonal equation evolution which controls how the support spreads (this was crucial in the decomposition Lemma \ref{lem:dec}).

As we can see in the proof of the main theorem (Theorem \ref{thm}), the estimate on the Hölder semi-norm depends in a non-linear fashion on the oscillation of $u$. This is a feature of the scaling of our equation. Linear estimates are expected for the Lipschitz semi-norm of the solution, because the Lipschitz scaling preserves the ellipticity constants of the equation. In order to establish these type of results we plan to pursue regularity estimates over the free boundary. See for instance \cite{Caffarelli_Friedman-1980} for a related approach in the case of divergence type equations.

Keeping in mind the interest established in \cite{Brandle_Vazquez-2005} on equations of the form $\p_tu = a(u)\D u+|Du|^2$ with $a$ sublinear as $u\to0^+$, we considered extending our result to $\p_tu = u^\a\D u+|Du|^\b$. The scaling of this equation seems quite restrictive. For instance, if our goal is to show a diminish of oscillation of order $\gamma$ in space we find the constrain $(\a+1)\gamma-2=\beta\gamma-\b$ on $\gamma$, unless $\a+1=\b=2$ which is our original case. Otherwise, we need to find an strategy to show directly a $\gamma$-Hölder modulus of continuity (in space) with $\gamma=(2-\beta)/(\a+1-\b)$, which is not necessarily a small exponent. For $\a>1$ (the case we can relate with \cite{Brandle_Vazquez-2005}) we get that $\gamma\to0^+$ if $\b\to2^-$, perhaps there is an opportunity to show a result in such range by a perturbative approach.

The two phase PME is also a well understood problem from the variational point of view since the eighties \cite{MR683353,MR654859,DiBenedetto-1982,DiBenedetto-1983,MR696738}. However, as far as we know, the viscosity solution approach for this particular type of two-phase free boundary problem has not been developed. Expanding the variational equation $\p_tu=\div(|u|Du)$, we see that the non-variational form should say instead the following: Each phase $u_\pm = \max(\pm u,0)$ satisfies the un-signed PME outside of the free boundary points connecting the two phases
\[
\p_t u = u\D u + |Du|^2 \text{ in } (\spt u_\pm \sm \spt u_\mp) \cap \W\times(t,s]
\]
and at the two-phase free boundary points we get
\[
\frac{\p_t u}{|Du_+|} = -\frac{\p_t u}{|Du_-|} = |Du_+|-|Du_-| \text{ in } \spt u_+ \cap \spt u_- \cap \W\times(t,s]
\]
In other words we can just say that
\[
\p_t u = |u|\D u + |Du_+|^2 - |Du_-|^2 \text{ in } \W\times(t,s].
\]
Besides the challenges that the well-posedness may present, we believe that our treatment for the regularity theory requires some substantial modifications. For instance, the ABP Lemma \ref{lem:abp} uses as test functions paraboloids that start growing from the zero level set. We could consider paraboloids that start from smaller level sets, however it is not clear how to use the eikonal equation to prevent the horizontal spread of the contact set whenever the contact happens at $\{u \in (-m,0]\}$. 

An interesting problem arises from models with non-local interactions. Let us recall that the continuity equation $\p_t u - \div(uDp) =0$ models the evolution of the density $u$ driven by the pressure $p$. The case $p=-(-\D)^{-\s}u$ with $\s\in(0,1)$ gives the fractional PME proposed by Caffarelli and Vázquez in \cite{Caffarelli_Vazquez-2011}. Meanwhile the regularity of the solution was established by the de Giorgi method in \cite{MR3082241}, the regularity of the free boundary remains largely open. The difficulty resides on the presence of the non-local drift term $Dp\cdot Du$ which prevents the comparison principle. Preliminary computations show that our technique could be extended to the equation
\[
\p_tu = -u(-\D)^{1-\s}u + |Du|^2
\]
where the dangerous presence of the non-local drift is replaced by the classical and local one, with the trade off of destroying the variational structure of the problem.

The regularity of the free boundary for evolution problems has a well developed non-variational approach, the book by Caffarelli and Salsa \cite{MR2145284} is our recommended source. For the PME a very elegant improvement of flatness strategy was recently implemented by Kienzler, Koch and Vázquez in \cite{kienzler2018flatness}. The regularity of a broader family of PME type equations may now open the possibility of extending this higher regularity for the free boundary to the corresponding non-variational class.

\bibliographystyle{plain}
\bibliography{mybibliography}

\begin{thebibliography}{10}

\bibitem{MR524760}
Donald~G. Aronson and Philippe B\'{e}nilan.
\newblock R\'{e}gularit\'{e} des solutions de l'\'{e}quation des milieux poreux
  dans {${\bf R}^{N}$}.
\newblock {\em C. R. Acad. Sci. Paris S\'{e}r. A-B}, 288(2):A103--A105, 1979.

\bibitem{Brandle_Vazquez-2005}
Cristina Br\"{a}ndle and Juan~Luis V\'{a}zquez.
\newblock Viscosity solutions for quasilinear degenerate parabolic equations of
  porous medium type.
\newblock {\em Indiana Univ. Math. J.}, 54(3):817--860, 2005.

\bibitem{Cabre-1997}
Xavier Cabr\'{e}.
\newblock Nondivergent elliptic equations on manifolds with nonnegative
  curvature.
\newblock {\em Comm. Pure Appl. Math.}, 50(7):623--665, 1997.

\bibitem{Caffarelli_Cabre-1995}
Luis Caffarelli and Xavier Cabr\'{e}.
\newblock {\em Fully nonlinear elliptic equations}, volume~43 of {\em American
  Mathematical Society Colloquium Publications}.
\newblock American Mathematical Society, Providence, RI, 1995.

\bibitem{MR683353}
Luis Caffarelli and L.~Craig Evans.
\newblock Continuity of the temperature in the two-phase {S}tefan problem.
\newblock {\em Arch. Rational Mech. Anal.}, 81(3):199--220, 1983.

\bibitem{Caffarelli_Friedman-1979}
Luis Caffarelli and Avner Friedman.
\newblock Continuity of the density of a gas flow in a porous medium.
\newblock {\em Trans. Amer. Math. Soc.}, 252:99--113, 1979.

\bibitem{Caffarelli_Friedman-1980}
Luis Caffarelli and Avner Friedman.
\newblock Regularity of the free boundary of a gas flow in an {$n$}-dimensional
  porous medium.
\newblock {\em Indiana Univ. Math. J.}, 29(3):361--391, 1980.

\bibitem{MR2145284}
Luis Caffarelli and Sandro Salsa.
\newblock {\em A geometric approach to free boundary problems}, volume~68 of
  {\em Graduate Studies in Mathematics}.
\newblock American Mathematical Society, Providence, RI, 2005.

\bibitem{MR3082241}
Luis Caffarelli, Fernando Soria, and Juan~Luis V\'{a}zquez.
\newblock Regularity of solutions of the fractional porous medium flow.
\newblock {\em J. Eur. Math. Soc. (JEMS)}, 15(5):1701--1746, 2013.

\bibitem{Caffareli_Vazquez-1999}
Luis Caffarelli and Juan~Luis Vázquez.
\newblock Viscosity solutions for the porous medium equation.
\newblock In {\em Differential equations: {L}a {P}ietra 1996 ({F}lorence)},
  volume~65 of {\em Proc. Sympos. Pure Math.}, pages 13--26. Amer. Math. Soc.,
  Providence, RI, 1999.

\bibitem{Caffarelli_Vazquez-2011}
Luis Caffarelli and Juan~Luis Vázquez.
\newblock Nonlinear porous medium flow with fractional potential pressure.
\newblock {\em Arch. Ration. Mech. Anal.}, 202(2):537--565, 2011.

\bibitem{MR1118699}
Michael~G. Crandall, Hitoshi Ishii, and Pierre-Louis Lions.
\newblock User's guide to viscosity solutions of second order partial
  differential equations.
\newblock {\em Bull. Amer. Math. Soc. (N.S.)}, 27(1):1--67, 1992.

\bibitem{MR1623198}
P.~Daskalopoulos and R.~Hamilton.
\newblock Regularity of the free boundary for the porous medium equation.
\newblock {\em J. Amer. Math. Soc.}, 11(4):899--965, 1998.

\bibitem{DiBenedetto-1982}
Emmanuele DiBenedetto.
\newblock Continuity of weak solutions to certain singular parabolic equations.
\newblock {\em Ann. Mat. Pura Appl. (4)}, 130:131--176, 1982.

\bibitem{DiBenedetto-1983}
Emmanuele DiBenedetto.
\newblock Continuity of weak solutions to a general porous medium equation.
\newblock {\em Indiana Univ. Math. J.}, 32(1):83--118, 1983.

\bibitem{DiBenedetto_Urbano_Vespri-2004}
Emmanuele DiBenedetto, José~Miguel Urbano, and Vicenzo Vespri.
\newblock Current issues on singular and degenerate evolution equations.
\newblock In {\em Evolutionary equations. {V}ol. {I}}, Handb. Differ. Equ.,
  pages 169--286. North-Holland, Amsterdam, 2004.

\bibitem{MR1814364}
David Gilbarg and Neil~S. Trudinger.
\newblock {\em Elliptic partial differential equations of second order}.
\newblock Classics in Mathematics. Springer-Verlag, Berlin, 2001.
\newblock Reprint of the 1998 edition.

\bibitem{Gurtin-1977}
Morton~E. Gurtin and Richard~C. MacCamy.
\newblock On the diffusion of biological populations.
\newblock {\em Math. Biosci.}, 33(1-2):35--49, 1977.

\bibitem{kienzler2018flatness}
Clemens Kienzler, Herbert Koch, and Juan~Luis V{\'a}zquez.
\newblock Flatness implies smoothness for solutions of the porous medium
  equation.
\newblock {\em Calculus of Variations and Partial Differential Equations},
  57(1):18, 2018.

\bibitem{MR2600689}
Inwon~C. Kim and Helen~K. Lei.
\newblock Degenerate diffusion with a drift potential: a viscosity solutions
  approach.
\newblock {\em Discrete Contin. Dyn. Syst.}, 27(2):767--786, 2010.

\bibitem{MR2763347}
Inwon~C. Kim and Norbert Po\v{z}\'{a}r.
\newblock Viscosity solutions for the two-phase {S}tefan problem.
\newblock {\em Comm. Partial Differential Equations}, 36(1):42--66, 2011.

\bibitem{MR3116010}
Inwon~C. Kim and Norbert Po\v{z}\'{a}r.
\newblock Nonlinear elliptic-parabolic problems.
\newblock {\em Arch. Ration. Mech. Anal.}, 210(3):975--1020, 2013.

\bibitem{Krylov-1979}
Nikolai~V. Krylov and Mikhail~V. Safonov.
\newblock An estimate for the probability of a diffusion process hitting a set
  of positive measure.
\newblock {\em Dokl. Akad. Nauk SSSR}, 245(1):18--20, 1979.

\bibitem{Krylov-1980}
Nikolai~V. Krylov and Mikhail~V. Safonov.
\newblock A property of the solutions of parabolic equations with measurable
  coefficients.
\newblock {\em Izv. Akad. Nauk SSSR Ser. Mat.}, 44(1):161--175, 239, 1980.

\bibitem{Mooney-2015}
Connor Mooney.
\newblock Harnack inequality for degenerate and singular elliptic equations
  with unbounded drift.
\newblock {\em J. Differential Equations}, 258(5):1577--1591, 2015.

\bibitem{padron-2004}
V\'{\i}ctor Padr\'{o}n.
\newblock Effect of aggregation on population recovery modeled by a
  forward-backward pseudoparabolic equation.
\newblock {\em Trans. Amer. Math. Soc.}, 356(7):2739--2756, 2004.

\bibitem{MR696738}
Paul~E. Sacks.
\newblock Continuity of solutions of a singular parabolic equation.
\newblock {\em Nonlinear Anal.}, 7(4):387--409, 1983.

\bibitem{urbano-2008}
José~Miguel Urbano.
\newblock {\em The method of intrinsic scaling}, volume 1930 of {\em Lecture
  Notes in Mathematics}.
\newblock Springer-Verlag, Berlin, 2008.

\bibitem{vazquez-2007}
Juan~Luis V\'{a}zquez.
\newblock {\em The porous medium equation}.
\newblock Oxford Mathematical Monographs. The Clarendon Press, Oxford
  University Press, Oxford, 2007.

\bibitem{MR3158522}
Yu~Wang.
\newblock Small perturbation solutions for parabolic equations.
\newblock {\em Indiana Univ. Math. J.}, 62(2):671--697, 2013.

\bibitem{Wu-1993}
Lang-Fang Wu.
\newblock A new result for the porous medium equation derived from the {R}icci
  flow.
\newblock {\em Bull. Amer. Math. Soc. (N.S.)}, 28(1):90--94, 1993.

\bibitem{MR654859}
William~P. Ziemer.
\newblock Interior and boundary continuity of weak solutions of degenerate
  parabolic equations.
\newblock {\em Trans. Amer. Math. Soc.}, 271(2):733--748, 1982.

\end{thebibliography}

\end{document}